\documentclass[11pt]{amsart}

\usepackage{pb-diagram}
\usepackage{calrsfs}
\usepackage{amssymb}
\usepackage{amsmath}
\usepackage{amsxtra}
\usepackage{amscd}
\usepackage{mathtools}
\usepackage{float}
\usepackage{graphicx} 
\usepackage{amsfonts}
\usepackage[utf8]{inputenc}
\usepackage[T1]{fontenc}
\usepackage{enumerate}
\usepackage{array,multirow}
\usepackage{color}
\usepackage[all]{xy}
\usepackage{cases}
\usepackage{empheq}
\usepackage{todonotes}
\usepackage{amsbsy}
\usepackage{soul}
\usepackage{cancel}
\usepackage{calc}
\usepackage{mathdots}
\usepackage{framed}
\usepackage{color, colortbl}
\definecolor{LightCyan}{rgb}{0.88,1,1}
\definecolor{Gray}{gray}{0.9}
\usetikzlibrary{cd, fit}
\tikzset{myboxgroup/.style={draw, densely dotted}} 
\usepackage{pgfplots}
\pgfplotsset{compat=1.13}
\usepackage{algorithmicx}
\usepackage{algorithm}
\usepackage{algpseudocode}
\usepackage{hyperref}
\usepackage{tikz}
\usetikzlibrary{arrows, decorations.markings}
\usetikzlibrary{cd}

\newsavebox\CBox
\newcommand\hcancel[2][0.5pt]{%
  \ifmmode\sbox\CBox{$#2$}\else\sbox\CBox{#2}\fi%
  \makebox[0pt][l]{\usebox\CBox}%
  \rule[0.5\ht\CBox-#1/2]{\wd\CBox}{#1}}

\newtheorem{theorem}{Theorem}

\newtheorem{lemma}{Lemma}

\newtheorem{remark}{Remark}

\newcommand{\Kh}{\mathrm{Kh}}
\newcommand{\N}{\mathbb{N}}
\newcommand{\Z}{\mathbb{Z}}
\newcommand{\Q}{\mathbb{Q}}
\newcommand{\C}{\mathbb{C}}
\newcommand{\qdim}{\mathrm{qdim}}

\numberwithin{equation}{section}

\begin{document} 

\title[A generalized skein relation for Khovanov homology and a categorification  of $\theta$
]{A generalized skein relation for Khovanov homology 
and a categorification of the $\theta$-invariant
}

\author{M. Chlouveraki}
\address{Laboratoire de Math\'ematiques, UVSQ / Universit\'e Paris-Saclay,
B\^atiment Fermat,
45 avenue des Etats-Unis,
78035 Versailles cedex, France}
\email{maria.chlouveraki@uvsq.fr}
\urladdr{http://chlouveraki.perso.math.cnrs.fr/}

\author{D. Goundaroulis}
\address{Center for Integrative Genomics,
University of Lausanne,
CH-1015, Lausanne, Switzerland.
\\ Swiss Institute of Bioinformatics, CH-1015, Lausanne, Switzerland.}
\email{dimoklis.gkountaroulis@unil.ch}

\author{A. Kontogeorgis}
\address{Department of Mathematics, University of Athens\\
Panepistimioupolis, 15784 Athens, Greece
}
\email{kontogar@math.uoa.gr}

\author{S. Lambropoulou}
\address{ Department of Mathematics,
National Technical University of Athens,
Zografou campus, GR--157 80 Athens, Greece.}
\email{sofia@math.ntua.gr}
\urladdr{http://www.math.ntua.gr/~sofia}


\thanks{The first author is supported by the \emph{Agence Nationale de la Recherche} through the JCJC project ANR-18-CE40-0001.}



\begin{abstract}
The Jones polynomial is a famous link invariant that can be defined diagrammatically via a skein relation.
Khovanov homology is a richer link invariant that categorifies the Jones polynomial. Using spectral sequences, we obtain a skein-type relation satisfied by the Khovanov homology. Thanks to this relation, we are able to generalize the Khovanov homology  in order to obtain a categorification of the $\theta$-invariant, which is itself a generalization of the Jones polynomial.
\end {abstract}
\maketitle

\section{Introduction}

One of the greatest achievements in knot theory and low-dimensional topology is the pioneering construction of the Jones polynomial by V.~F.~R.~Jones in 1984, which
 advanced spectacularly the tabulation of knots.
The value of the Jones polynomial $J(L)$ on an oriented link $L$ can be calculated through different methods. 
One of them is algebraic and consists of computing the Markov trace of the image of a braid representative of $L$ in the Temperley--Lieb algebra.
Another one is diagrammatic and
uses the fact that the Jones polynomial satisfies a ``skein relation'', that is, a linear relation between the values of the polynomial on a collection of three links that differ from each other only on a selected crossing. More specifically, for an indeterminate $q$, we have
$$q^{-2}J(L_+) - q^2J(L_-)=(q^{-1}-q)J(L_0)$$
where $L_+,L_-,L_0$ is a so-called Conway triple, or equivalently,
$$
q^{-2}J( 
\raisebox{-0.35em}{\includegraphics[scale=0.5]{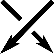}} ) - q^2
J(
\raisebox{-0.35em}{\includegraphics[scale=0.5]{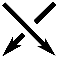}} ) = (q^{-1}-q)
J(\raisebox{-0.35em}{\includegraphics[scale=0.5]{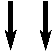}}).
$$
Combining this skein relation with the initial condition that the value of the Jones polynomial on the unknot is equal to $1$ allows us to compute the Jones polynomial of any link. 

The Framization of the Temperley--Lieb algebra, introduced in \cite{GJKL}, is a non-trivial extension of the classical Temperley--Lieb algebra via the addition of the so-called ``framing'' generators, each of which is a generator of a cyclic group of order $d$. 
It is also endowed with a Markov trace, which gives rise to an invariant of framed links. When restricted to classical links, this invariant is denoted by $\theta$.
It has been shown in \cite{CJKL,GoLa} that the $\theta$-invariant can be also defined diagrammatically through a skein relation, which is however not global as in the case of the Jones polynomial. In fact, the same skein relation as the one satisfied by the Jones polynomial holds, but only on crossings \emph{between different components}; we call these crossings \emph{mixed crossings}.
Using a recursive proof method developed originally by Lickorish--Millett \cite{Lickorish1987-am} and adapted for $\theta$ by Kauffman--Lambropoulou \cite{Kauffman2017-ra}, it can be shown that calculating the value of $\theta$ on a link $L$ amounts to calculating the value of $\theta$ on links that are unions of unlinked knots and are obtained via the skein relation. More specifically, a series of switchings and smoothings of mixed crossings transforms the initial link to a family of unions of unlinked knots, 
 called \emph{descending stacks}. Therefore, as we also explain in Section \ref{sec:knottheory-invariant},
the computation of $\theta$ can be seen as a dynamical system whose initial condition is the value of $\theta$ on a union of unlinked knots: if $L$ is union of $r$ unlinked knots, we have
$$\theta(L) = d^{r-1} J(L).$$
The invariant $\theta$ is stronger than the Jones polynomial, in the sense that it distinguishes links that the Jones polynomial cannot distinguish \cite{CJKL, GoLa, Chlou}.

Khovanov homology, introduced by M.~Khovanov \cite{MR1740682}, is an oriented link invariant that arises as the homology of a chain complex associated with a link. It is regarded as a categorification of the Jones polynomial, which can be obtained as the graded Euler characteristic of this complex.  However, Khovanov homology encompasses more information about a link than its Jones polynomial. Further, it is known that Khovanov homology can detect the unknot, while it is still an open question whether the Jones polynomial can do the same.

Inspired by this and motivated by the fact that the $\theta$-invariant is stronger than the Jones polynomial, we expect that a categorification of the $\theta$-invariant would be stronger than all the invariants mentioned above. Given the skein-theoretic definition of $\theta$, a first step towards a  possible categorification is the obtention of a skein relation satisfied by the Khovanov homology.  Unfortunately, the Khovanov homology does not satisfy a skein relation in the usual sense \cite{Wehrli2003-jj}. So in the first part of our article we investigate a  skein-type relation for Khovanov homology, using the machinery of spectral sequences. We show that the Khovanov homology (more specifically, its Poincar\'e polynomial) satisfies a generalized skein relation which involves its values on a Conway triple, plus a ``defect'' term that arises from spectral sequences associated to the Khovanov chain complex.

In the second part of our article, we aim to define a ``framization'' of the Khovanov homology, which should have the $\theta$-invariant as Euler characteristic. 
On a union of unlinked knots, we achieve this by tensoring with the group algebra of the cyclic group of order $d$. On an arbitrary link $L$, 
we follow the method of Lickorish--Millett and Kauffman--Lambropoulou. However, their algorithm  depends on several choices made on the link, such as, for instance, the ordering of the knot components and of the mixed crossings. So one has to show that the final result is independent of these choices.  This is achieved in \cite{Lickorish1987-am}  and in \cite{Kauffman2017-ra} thanks to the skein relation and  the properties of the base invariant involved. In our case though, we do not have a skein relation in the classical sense and we cannot prove the independence of the choices made in the same
way that they did. This is why we slightly modify the generalized skein relation proved in the first part of the article, before we apply it only to mixed crossings.
 By summing over several choices and dividing by their number, we eventually obtain a link  invariant $\Kh_{d,d'}$, depending on $d$ and an extra parameter
 $d'$, which is the Poincar\'e polynomial of a homology $KH^{*,*}_{d,d'}(L)$, that is,
\[
\Kh_{d,d'}(L)= \sum_{i,j\in \Z} t^i q^j \dim KH_{d,d'}^{i,j}(L).
\]
We prove that both the classical Khovanov homology and the $\theta$-invariant can be obtained as specializations of $\Kh_{d,1}$; the former for $d=1$, the latter for $t=-1$.

 \section{A generalized skein relation for Khovanov homology}

Let $q$ and $t$ be indeterminates over $\Q$.

\subsection{Generalities on graded vector spaces}

Let $V = \oplus_m V^m$ be a graded $\Q$-vector space. The \emph{graded dimension}
$\qdim(V)$ of $V$ is the Laurent polynomial in $q$ given by 
$$\qdim(V)=\sum_{m} q^m {\rm dim}(V^m).$$
If $V'$ is another graded vector space, the graded dimension satisifes
$$\qdim(V \otimes V') = \qdim(V)\qdim(V') \quad \text{and}\quad
\qdim(V \oplus V') = \qdim(V)+\qdim(V').$$

Let $l \in \mathbb{Z}$. We define the  $l$-\emph{shift} $V\{l\}$ of $V$ to be the graded vector space defined by 
$$V\{l\}^m=V^{m-l}.$$
Note that $\qdim(V\{l\})=q^l\qdim(V)$.

\subsection{Classical Khovanov homology} 

In this section, we give a short introduction to Khovanov homology following the exposition of \cite{Turner2006-ck}.
Let $L$ be a link and $D$ a diagram of $L$ with $n$ crossings.  
Each crossing can be resolved in two ways: 

 \begin{figure}[h!]
 \includegraphics[scale=0.2]{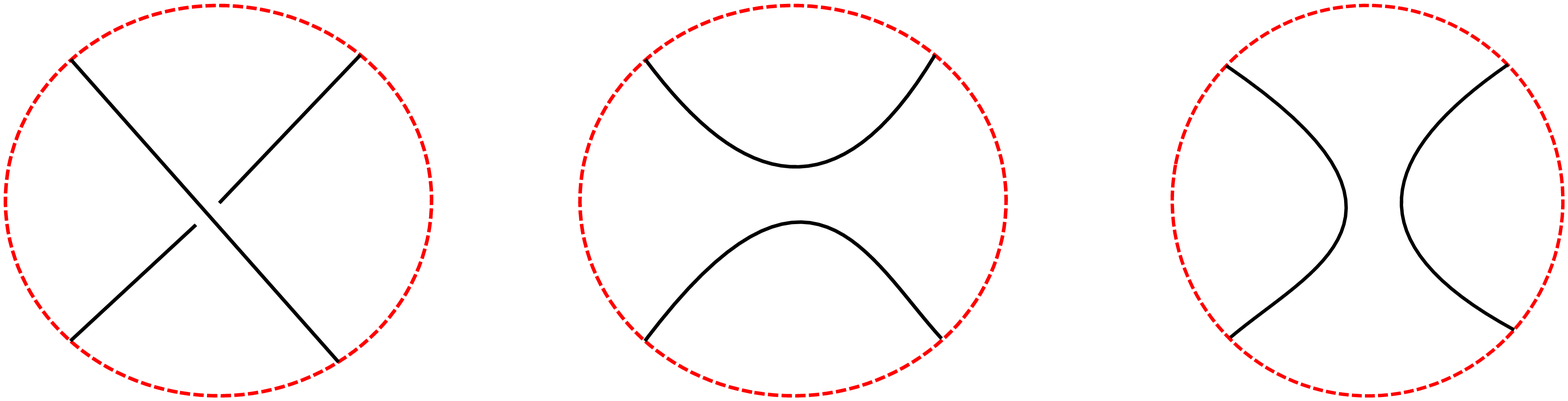}
 \caption{Smoothings of a crossing \label{pictureP2}}
 \end{figure}
 The first resolution in Figure 1 is called a $0$-\emph{smoothing} and the second one is an $1$-\emph{smoothing}.
Thus, there are $2^n$ ways of resolving all crossings of $D$, each of them resulting to a collection of circles in the plane.
If now we number the crossings of $D$ by $1,2,\ldots,n$, then each such collection can be represented by a binary $n$-string where the entry $0$ or $1$ in the $j$-th position corresponds to a $0$-smoothing or $1$-smoothing respectively of the $j$-th crossing, for all $j=1,\ldots,n$. In the end, $D$ has $2^n$ smoothings indexed by $I_n:=\{0,1\}^n$. We can thus form a hypercube as in the following figure (for $n$=3)
$$ 
\xymatrix{
   &  100  \ar@{-}[dr] \ar@{-}[r] & 110 \ar@{-}[dr] & \\
000 \ar@{-}[ur] \ar@{-}[dr] \ar@{-}[r]  & 010 \ar@{.}[ur] \ar@{.}[dr]& 101 \ar@{-}[r] & 111 \\
   & 001 \ar@{-}[ur] \ar@{-}[r] & 011  \ar@{-}[ur] & 
}
$$
with an edge between words differing in exactly one place.
For general $n$, we see the smoothings as vertices of the hypercube indexed by 
$I_n$. For $\alpha \in I_n$, we will denote by $r_\alpha$ the number of 1's in $\alpha$ and by $k_\alpha$ the number of circles in the plane of the associated smoothing.

In Khovanov homology we further assume that the link $L$ is oriented. 
We denote by $n_+=n_+(D)$ the number of positive crossings and by $n_-=n_-(D)$ the number of negative crossings, and we use the simpler notation $n_+$ and $n_-$ whenever the diagram is fixed.
Let $V$ be a $2$-dimensional $\Q$-vector space with basis $\{e,x\}$. 
 We grade the two basis elements by ${\rm deg}(e)=1$ and ${\rm deg}(x)=-1$.
 To each  $\alpha \in I_n$ we associate the graded vector space
 $$V_\alpha:=V^{\otimes k_\alpha}\{r_\alpha+n_+-2n_-\}.$$
 For $i \in \{-n_-,\ldots,n_+\}$, we define $C^{i,*}(D)$ to be the direct sum
 \begin{equation}\label{def of C}
 C^{i,*}(D):= \bigoplus_{\alpha \in I_n \text{ with } r_\alpha=i+n_-} V_\alpha.
 \end{equation}
 For $v \in V_\alpha \subset C^{*,*}(D)$, we have $v \in C^{i,j}(D)$ if and only if  $i=r_\alpha-n_-$ and 
  $j= {\rm deg}(v)+r_\alpha+n_+-2n_-$.
  We then say that $v$ has 
  \emph{homological grading} $i$ and \emph{quantum grading} 
  $j$.
 
We will now define a differential $d$ turning $(C^{*,*}(D),d)$ into a complex. 
First, observe that every edge of the hypercube is between elements of $I_n$ differing in exactly one place. It can thus be transformed into an arrow from the string with $0$ in that place towards the string with $1$ in that place. This arrow can be labelled by a string that is the same as the ones labelling the tail and the head except that it has a $\star$ in the position that changes. For example, there is an arrow
from $0100$ to $0110$ which is denoted by $01\star0$.

Now note that, for an arrow $\zeta: \alpha \rightarrow \alpha'$, the smoothings $\alpha$ and $\alpha'$ are identical except for a small disc, the \emph{changing disc}, around the crossing that changes from a $0$-smoothing to an $1$-smoothing (the one marked by a $\star$ in the label of $\zeta$). Since each circle in a smoothing has a copy of the vector space $V$ attached to it, we can define a linear map $d_\zeta: V_\alpha \rightarrow V_{\alpha'}$ as follows:
Let $m: V \otimes V \rightarrow V$ be the linear map defined by
$$m(e \otimes e)=e,\quad m(e \otimes x) = m(x \otimes e)=x,\quad m(x \otimes x)=0$$
and let 
$\Delta : V \rightarrow V \otimes V$ be the linear map defined by
$$\Delta(e)=e \otimes x + x \otimes e, \quad \Delta(x)=x\otimes x.$$
Then $d_\zeta$ is defined to be the identity on circles not entering the changing disc and either $m$ or $\Delta$ on the circles appearing in the changing disc (depending on whether two circles are fused into one or a circle splits into two when going from $\alpha$ to $\alpha'$).

Finally, we define a map $d^i: C^{i,*}(D) \rightarrow C^{i+1,*}(D)$ by setting
$$d^i(v):= \sum_{\zeta \text{ with } {\rm Tail}(\zeta)=\alpha} {\rm sign}(\zeta)d_\zeta(v)
\quad \text{for all } v \in V_\alpha \subset C^{i,*}(D)$$
where ${\rm sign}(\zeta)=(-1)^{\text{$\#$ of 1's to the left of $\star$ in $\zeta$}}$.
We have $d^{i+1} \circ d^i=0$.

The graded Euler characteristic of this complex, \emph{i.e.,} 
$$\sum_i (-1)^i\qdim(C^{i,*}(D)) \in \Q[q,q^{-1}]$$
is the unnormalized Jones polynomial $\hat{J}(L)$ of $L$. Dividing $\hat{J}(L)$  by $q+q^{-1}$ yields the Jones polynomial  $J(L)$ of $L$.

We define the \emph{Khovanov homology of the diagram $D$} by
$$KH^{*,*}(D):=H(C^{*,*}(D),d).$$ 
The Khovanov homology is invariant under the Reidemeister moves, and thus a link invariant. Hence, it does not depend on the choice of diagram $D$, and we can talk about the \emph{Khovanov homology of the link $L$}.  We have
$$\sum_i (-1)^i\qdim(KH^{i,*}(L)) =\hat{J}(L)= (q+q^{-1})J(L).$$

The Khovanov homology of $L$ can be also be read off the \emph{Khovanov polynomial} $\mathrm{Kh}(L)$ of $L$, which is given by
$$\mathrm{Kh}(L) =  \sum_{i,j} t^i q^j {\rm dim}(KH^{i,j}(L)) =\sum_i t^i\qdim(KH^{i,*}(L)).$$
Obviously, evaluating $\Kh(L)$ at $t=-1$ yields the unnormalized Jones polynomial $\hat{J}(L)$.

We choose not to give right now explicit examples of computations of the classical Khovanov homology, because: (a) some will be given in \S \ref{sec-Examples} along with further computations; (b) the reader can find plenty in literature (see, for example, \cite{Natan2002} or \cite{Turner2006-ck}).

\subsection{The effect of switching a crossing}

A skein relation relates the 
value of a link invariant of a diagram $D$ of a link $L$ with the values of the diagrams obtained by 
switching and smoothing a given crossing. 
Let us consider the numbering of the crossings of $D$ and the associated hypercube introduced in the previous section. If we switch  the $j$-th crossing the sign from negative to positive, then every $1$ at the $j$-th position should become $0$ and every $0$ should change to $1$, as if we take the “not” operator on a given binary number at the $j$-th position. In general the operation of switching the crossing induces an automorphism of the hypercube.

For example, if we take $j=2$, then the hypercube
$$
\xymatrix{
   &  100  \ar@{-}[dr] \ar@{-}[r] & 110 \ar@{-}[dr] & \\
000 \ar@{-}[ur] \ar@{-}[dr] \ar@{-}[r]  & 010 \ar@{.}[ur] \ar@{.}[dr]& 101 \ar@{-}[r] & 111 \\
   & 001 \ar@{-}[ur] \ar@{-}[r] & 011  \ar@{-}[ur] & 
}
$$
should  change to 
\begin{center}
\includegraphics[scale=1.3]{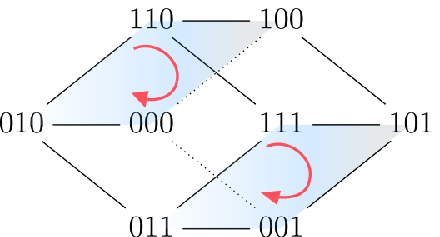}
\end{center}
Notice that in order to map the new transformed cube to the Khovanov cube with this changed position we have to rotate the two highlighted faces together with the maps at the edges. For a detailed example explaining this rotation after switching a crossing from negative to positive we refer to the example of the Hopf link in \S \ref{HopfExample}. 

\subsection{A generalized skein relation}

Consider two link diagrams $D^+$ and $D^-$ that differ on exactly one crossing, which is positive in the former and negative in the latter. Equivalently, we can see this crossing as the one  we choose to switch, thus obtaining one diagram from the other.
In order to understand the relation between $D^+$ and $D^-$ we will study the relation of $D^+$ (respectively $D^-$) to the two resolutions $D^+_0$ and $D^+_1$ (respectively $D^-_0$ and $D^-_1$), obtained respectively by the $0$-smoothing and the $1$-smoothing of the selected crossing, as in the following picture:

\vskip 0.5truecm
\begin{center}
\includegraphics[scale=0.5]{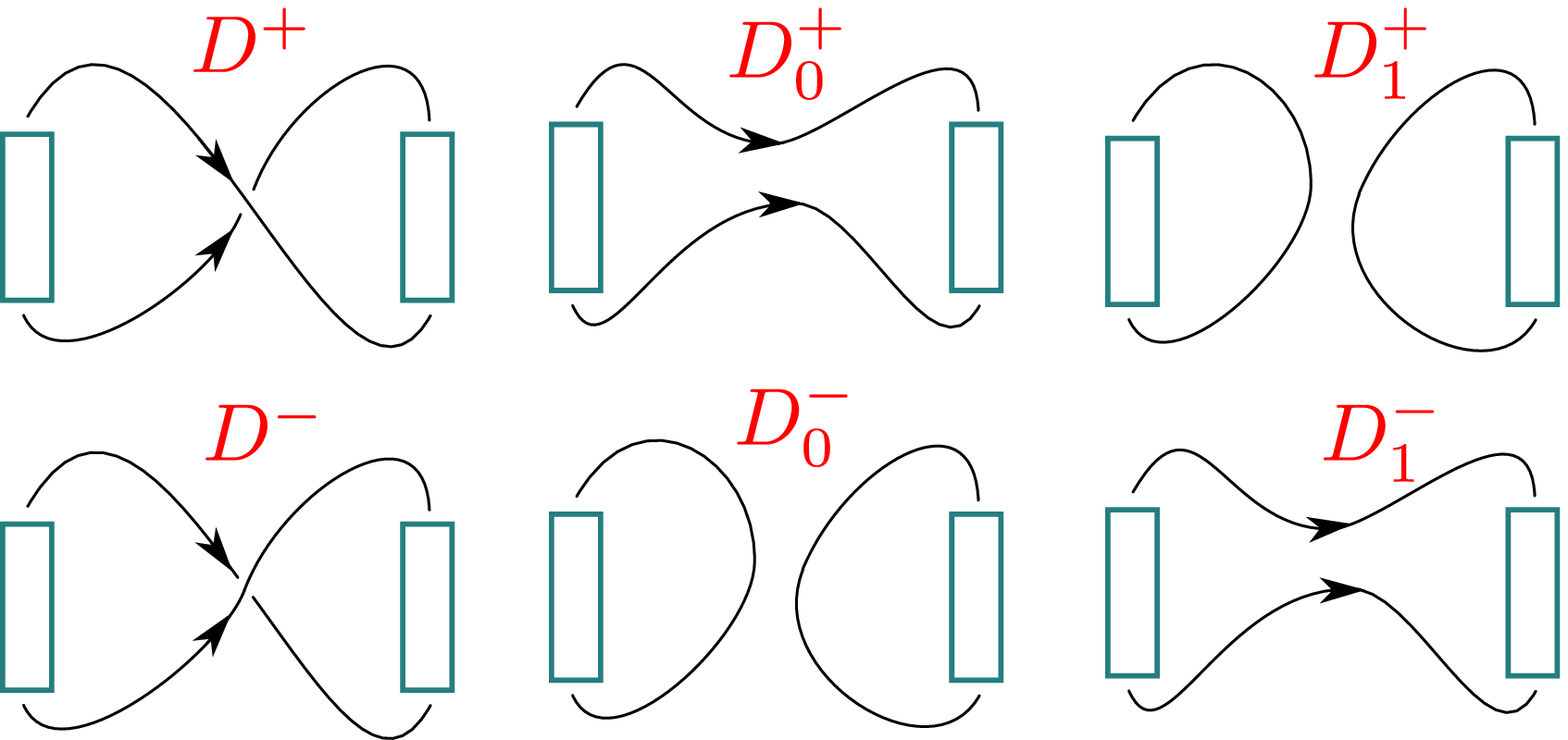}
\end{center}

Following \cite[Chapter 3]{Turner2006-ck}, we distinguish the following cases:\\
{\bf Case I:} the selected crossing is positive.
The diagram $D_0^+$
inherits an orientation from $D^+$ and for  $D_1^+$
we select an arbitrary orientation.  
Set 
\[
c^+:=\# \{\text{negative crossings in } D_1^+ \} - \# \{\text{negative crossings in } D^+\}.
\]
For each $j$ there is a short exact sequence:
\[
0 \rightarrow C^{i-c^+-1,j-3c^+-2}(D_1^+) \rightarrow C^{i,j}(D^+) 
\rightarrow 
C^{i,j-1}(D_0^+)
\rightarrow 0.
\]
{\bf Case II:} the selected crossing is negative. In this case
$D_1^-$
inherits its orientation from $D^-$
and for $D_0^-$
we select the same orientation as for $D_1^+$. 
Set 
\[
c^-:=\# \{\text{negative crossings in } D_0^{-} \} - \# \{\text{negative crossings in } D^-\}.
\]
For each $j$ there is a short exact sequence:
\[
0 \rightarrow C^{i,j+1}(D_1^-) \rightarrow {}
C^{i,j}(D^-)
\rightarrow 
C^{i-c^-,j-3c^- - 1}(D_0^-) 
\rightarrow 0.
\]

Now, we observe that $D_0^+ = D_1^-$ and $D_0^-=D_1^+$, after choosing the same orientation for the second pair.
We thus have
$c^+=c^-+1$,
and so
\[
 (i-2)-c^-=i-c^+-1 \quad \text{and}\quad (j-4)-3c^- -1=j-3c^- -5=j-3c^+-2.
 \]
Therefore,
the two short exact sequences above can be combined in a long one:  
\[
\xymatrix@C-=0.5cm{
0  \ar[r] & C^{i-2,j-3}(D_0^+) \ar[r]^{\psi_1} & 
C^{i-2,j-4}(D^-) \ar[r]^-{\phi_1}
  &  
C^{(i-2)-c^-,(j-4)-3c^- - 1}(D_0^-)
  \ar[r]& 0 
  \\
        0 \ar[r]& 
C^{i-c^+-1,j-3c^+-2}(D_0^-)
 \ar[r]^-{\phi_2}   \ar@{=}[urr] &
C^{i,j}(D^+)  \ar[r]^{\psi_2} & 
C^{i,j-1}(D_0^+) \ar[r]
& 0 
}
\]
\begin{lemma}
The above  exact sequence can be compactified to a 4-term exact sequence:
\begin{equation}
\label{combined}
\xymatrix@C-=0.6cm{
  0  \ar[r] & C^{i-2,j-3}(D_0^+) \ar[r]^{\psi_1} & 
C^{i-2,j-4}(D^-) \ar[r]^-{\phi} & 
C^{i,j}(D^+)  \ar[r]^{\psi_2} & 
C^{i,j-1}(D_0^+) \ar[r]
& 0  
}
\end{equation}
where $\phi=\phi_2 \circ \phi_1$.
\end{lemma}
\begin{proof}
Observe that 
$$
\ker \phi  =\{x \in C^{i-2,j-4}(D^-)\,|\, \phi_2 (\phi_1(x))=0\} = \{ x \in C^{i-2,j-4}(D^-)\,|\, \phi_1(x)=0\} = \ker \phi_1=\mathrm{im}\,\psi_1.
$$
Similarly, 
$$
\mathrm{im}\,\phi =\phi_2 \left( \phi_1 (C^{i-2,j-4}(D^-))\right) = \phi_2(C^{(i-2)-c^-,(j-4)-3c^- - 1}(D_0^-))= \mathrm{im}\,\phi_2=\ker \psi_2.
$$
\end{proof}

Equation (\ref{combined}) gives rise to  a double complex which is not zero on a parallel band of four lines and also only if $i \in \{-n_-(D^-),\ldots, n_+(D^+)\}$
(recall the definition of $C^{i,*}$ in \eqref{def of C} and note that
$n_-(D^-) > n_-(D_0^+)=n_-(D^+)$ and $n_+(D^-) = n_+(D_0^+)<n_+(D^+)$), so it behaves like a first quadrant double complex. On this parallel band the quantum degrees $j-3,j-4,j,j-1$ appear. For every fixed $j$ we will form  a spectral sequence out of this double complex and we will  denote elements on the $n$-th page by $E_{n,j}^{i,j}$, so that the second subindex shows the dependence on the quantum degree. For a good introduction to spectral sequences, the reader may refer to \cite[Chapter 5]{Weibel}.

\[
\xymatrix{
        & 0 \ar[d] & 0 \ar[d] & 
       \\
\ar[r]	& C^{i-2,j-3}(D_0^+) \ar[r]  \ar[d] &
          C^{i-1,j-3}(D_0^+) \ar[r]  \ar[d] &
\\ 
\ar[r]	& C^{i-2,j-4}(D^-) \ar[r] \ar[d] & 
          C^{i-1,j-4}(D^-) \ar[r] \ar[d] &
\\
\ar[r]
	&  C^{i,j}(D^+) \ar[r] \ar[d] &
	   C^{i+1,j}(D^+) \ar[r] \ar[d] &
\\
\ar[r]
	& C^{i,j-1} (D_0^+) \ar[r] \ar[d] & 
	  C^{i+1,j-1} (D_0^+) \ar[r] \ar[d] &
\\
	&  0 & 0  
}
\]

For the creation of the first vertical page $E_{1,j}^{i,j}$, we compute Kernel over Image along each row.  This coincides with taking the Khovanov homology in each position. Thus the first page measures how far each horizontal sequence is from being exact. The vertical arrows are preserved, and we obtain: 
{\tiny
\[
\xymatrix@!C{ 
 0 \ar[d]      & 0 \ar[d] & 0 \ar[d] & 0 \ar[d] 
       \\
KH^{i-3,j-3}(D_0^+) \ar[d]^{\alpha_{i-3,j}}	& KH^{i-2,j-3}(D_0^+) \ar[d]^{\alpha_{i-2,j}} & 
          KH^{i-1,j-3}(D_0^+) \ar[d]^{\alpha^{i-1,j}}   & KH^{i,j-3}(D_0^+) \ar[d]^{a_{i,j}} 
\\ 
KH^{i-3,j-4}(D^-)  \ar[d]^{\beta_{i-3,j}}
	& KH^{i-2,j-4}(D^-) \ar[d]^{\beta_{i-2,j}}  &  
          KH^{i-1,j-4}(D^-) \ar[d]^{\beta_{i-1},j}  &
           KH^{i,j-4}(D^-) \ar[d]^{\beta_{i,j}} 
\\
KH^{i-1,j}(D^+) \ar[d]^{\gamma_{i-1,j}} 	&
  KH^{i,j}(D^+) \ar[d]^{\gamma_{i,j}} &  
	   KH^{i+1,j}(D^+)  \ar[d]^{\gamma_{i+1,j}} &
      KH^{i+2,j}(D^+)  \ar[d]^{\gamma_{i+2,j}} 
\\
 KH^{i-1,j-1} (D_0^+)  \ar[d]	& KH^{i,j-1} (D_0^+)  \ar[d] &  
	  KH^{i+1,j-1} (D_0^+)  \ar[d] & KH^{i+2,j-1} (D_0^+)  \ar[d] 
\\
	0 &  0 & 0   & 0  
}
\]
}
In the above diagram we will denote the functions 
\begin{equation}
\label{seqCom}
\xymatrix{
0 \ar[r] &
KH^{i-2,j-3}(D_0^+) \ar[r]^-{\alpha_{i-2,j}} &
KH^{i-2,j-4}(D^-) \ar[r]^-{\beta_{i-2,j}} &
KH^{i,j}(D^+) \ar[r]^-{\gamma_{i,j}} &  
KH^{i,j-1}(D_0^+) \ar[r] &
0},
\end{equation}
where the second subindex $j$ illustrates that we are on the 
spectral sequence $E_{*,j}^{*,*}$.

For the creation of the second page $E_{2,j}^{i,j}$ (Table  \ref{tab:E2}), we take Kernel over Image along each column.  Thus, the second page measures how far the sequences of the first page are from being exact. Now, in each position we no longer get a link invariant, so
we will  use the notation $D_0^{+t}$ and $D_0^{+b}$ to differentiate between the copy of $D_0^+$ appearing in the top line and the copy of $D_0^+$ appearing in the bottom line. We obtain arrows that go two positions down and one left. 
We have that:
\begin{align}
\label{Edefect}
E_{2,j}^{i-2,j-3}(D_0^{+t}) &=\mathrm{ker}\alpha_{i-2,j}  
\\ \nonumber
E_{2,j}^{i-2,j-4}(D^-) &=\frac{\mathrm{ker}\beta_{i-2,j}}{\mathrm{im}\alpha_{i-2,j}} 
\\ \nonumber
E_{2,j}^{i,j}(D^+) &=\frac{\mathrm{ker}\gamma_{i,j}}{\mathrm{im}\beta_{i-2,j}} 
\\ \nonumber
E_{2,j}^{i,j-1}(D_0^{+b}) &=\mathrm{coker}\gamma_{i,j}.
\end{align}

We now create the third page $E_{3,j}^{i,j}$ (Table  \ref{tab:E3}) by taking 
Kernel over Image along each sequence of the $E_2$-page. We obtain arrows that go 3 positions down and 2 left. 
By looking at the $E_4$-page, we immediately obtain that all positions in the $E_3$-page are exact. Moreover, we observe that
$E^{i-1,j}_{3,j}(D^+) \cong E^{i,j}_{3,j}(D^+) \cong E^{i-2,j-4}_{3,j}(D^-) \cong 0$, whence we deduce that the positions marked in grey in the $E_2$-page are exact.

\begin{table}
  \caption{The $E_2$-page; \label{tab:E2} the positions marked in grey are exact.}

\begin{tikzcd}[/tikz/cells={/tikz/nodes={shape=asymmetrical
  rectangle,text width=2.0cm,text height=1.6ex,text depth=0.3ex,align=center}}]
&         &  &0  \arrow[ddl]    & 0  \arrow[ddl]  
\\
&         &   & 0 \ar[ddl]  
       \\
 &    &
 E_{2,j}^{i-2,j-3}(D_0^{+t}) \arrow[ddl]
 & 
 E_{2,j}^{i-1,j-3}(D_0^{+t}) \arrow[ddl] 
\\ 
 & 
  &
  \colorbox{gray!20}{
$E_{2,j}^{i-2,j-4} (D^-)$ } \arrow[ddl]  
\\
 &  \colorbox{gray!20}{ $E_{2,j}^{i-1,j}(D^+)$ } \arrow[ddl]&
   \colorbox{gray!20}{$E_{2,j}^{i,j}(D^+)$} \arrow[ddl] 
\\
 &
 E_{2,j}^{i-1,j-1} (D_0^{+b})  \arrow[ddl]  
    \\
   0 &  0       
       \\
   0    &        
\end{tikzcd}
\end{table}

By looking at the second, third and fourth  pages we obtain the following exact sequences:
\begin{equation}
\label{ses-pages}
\xymatrix{
0 \ar[r] & 
E_{3,j}^{i-1,j-3}(D_0^{+t}) \ar[r] & 
E_{3,j}^{i-1,j-1}(D_0^{+b}) \ar[r] &0
&  \\
  0 \ar[r] & 
   E_{3,j}^{i-2,j-3}(D_0^{+t}) \ar[r]
  &E_{2,j}^{i-2,j-3}(D_0^{+t}) \ar[r]
  & 
  E_{2,j}^{i-1,j} (D^+) \ar[r] & 
  0
  \\
  0 \ar[r] &
  E_{2,j}^{i-2,j-4}(D^-) \ar[r] &
  E_{2,j}^{i-1,j-1}(D_0^{+b}) \ar[r] &
  E_{3,j}^{i-1,j-1}(D_0^{+b}) \ar[r] &
  0
}
\end{equation}

\begin{table}
\caption{The $E_3$-page; \label{tab:E3} all positions are exact.}
\begin{tikzpicture}[baseline= (a).base]
\node[scale=.5] (a) at (0,0){
\begin{tikzcd}[row sep=0.7cm, column sep=0.5cm, /tikz/cells={/tikz/nodes={shape=asymmetrical
  rectangle,text width=3.5cm,text height=3ex,text depth=2.5ex,align=center}}]
& &         &   &   & 0 \ar[dddll]  & 0 \ar[dddll] &  
 \\
 & &         &   &    & 0  \ar[dddll] & & 
\\
 & &         &   & 0  \ar[dddll]  & 0  \ar[dddll] & &
       \\
 &    &
 & 
E_{3,j}^{i-2,j-3}(D_0^{+t}) \ar[dddll]  & E_{3,j}^{i-1,j-3}(D_0^{+t}) \ar[dddll] 
 &  &
\\ 
 & & 
  &
E_{3,j}^{i-2,j-4} (D^-) \ar[dddll]  &  
  &
&&
\\
& & E_{3,j}^{i-1,j}(D^+) \ar[dddll] &
   E_{3,j}^{i,j}(D^+) \ar[dddll] &  
    &
     & &
\\
& E_{3,j}^{i-2,j-1} (D_0^{+b}) &
 E_{3,j}^{i-1,j-1} (D_0^{+b})  \ar[dddll]  &   &  
     &   & &
    \\
 & 0    &         &   &   &   & &
       \\
0 &  0    &         &   &   &   & &
        \\
0 &      &         &   &   &   & &
\end{tikzcd}
};
\end{tikzpicture}
\end{table}

From now on, abusing notation, we will write $\mathrm{Kh}(E_{n}(D^*))$ for
$\sum_{i,j} t^i q^j {\rm dim}(E^{i,j}_{n,j}(D^*))$. 
The short exact sequences above imply that
\begin{align}
tq^3 \mathrm{Kh}(E_{3}(D_0^{+t})) &=tq\mathrm{Kh}(E_{3}(D_0^{+b}))
\,,\label{subs3} \\
t \mathrm{Kh}(E_{2}(D^+)) &=t^2q^3\mathrm{Kh}(E_{2}(D_0^{+t}))
-t^2q^3  \mathrm{Kh}(E_{3}(D_0^{+t}))\,,\label{subs1} \\
t^2 q^4 \mathrm{Kh}(E_{2}(D^{-})) &=tq \mathrm{Kh}(E_{2}(D_0^{+b}))
- tq \mathrm{Kh}(E_{3}(D_0^{+b})).\label{subs2}
\end{align}
Now, Equation \eqref{seqCom} combined with \eqref{Edefect} yields
\begin{equation}
(t^2q^3 -q) \mathrm{Kh}(D_0^+)
 -
t^2q^4 
\mathrm{Kh}(D^-) +
\mathrm{Kh}(D^+)
=C\,,
\end{equation}
where
$$
C= 
t^2q^3  \mathrm{Kh}(E_{2}(D_0^{+t}))
- t^2q^4
\mathrm{Kh}(E_{2}(D^-)) +
\mathrm{Kh}(E_{2}(D^+))
-q \mathrm{Kh}(E_{2}( D_0^{+b})) .
$$
Note again that $\mathrm{Kh}(E_{2}(D_0^{+t}))$ is not necessarily equal to $\mathrm{Kh}(E_{2}(D_0^{+b}))$. Substituting $\mathrm{Kh}(E_{2}(D^+))$ and $\mathrm{Kh}(E_{2}(D^-))$ in the above formula with the use of \eqref{subs1} and \eqref{subs2}
respectively, we obtain
$$C= 
(t^2q^3+tq^3)\mathrm{Kh}(E_{2}(D_0^{+t}))
-(q+tq)\mathrm{Kh}(E_{2}( D_0^{+b}))
-tq^3 \mathrm{Kh}(E_{3}(D_0^{+t})) +tq \mathrm{Kh}(E_{3}(D_0^{+b}))$$
The last two terms cancel out because of \eqref{subs3}, and we conclude that
\begin{equation}
\label{Cdefinit}
C= ( t+1)tq^3\mathrm{Kh}(E_2(D_0^{+t}))
-(t+1)q\mathrm{Kh}(E_2(D_0^{+b}))=:C(D_0^+,D^-,D^+) .
\end{equation}
We have thus proved the main result of this section:

\begin{theorem}
\label{skeinth}
The Khovanov polynomial $\mathrm{Kh}$ satisfies the generalized skein relation
\begin{equation}
\label{skein-rel-bar}
\boxed{
(t^2q^3 -q) \mathrm{Kh}(D_0^+)
 -
t^2q^4 
\mathrm{Kh}(D^-) +
\mathrm{Kh}(D^+)=C(D_0^+,D^-,D^+)
}
\end{equation}
where $C(D_0^+,D^-,D^+)$ is given by Equation \eqref{Cdefinit}. The generalized skein relation can be also written more symmetrically as
\begin{equation}
\label{skein-rel-bar2}
\boxed{
t^{-1}q^{-2}\mathrm{Kh}(D^+)
 -
tq^2 
\mathrm{Kh}(D^-) 
=(t^{-1}q^{-1}-tq ) \mathrm{Kh}(D_0^+) +C^{\rm sym}(D_0^+,D^-,D^+)
}
\end{equation}
where
\begin{equation}
C^{\rm sym}(D_0^+,D^-,D^+) =  (t+1)q\mathrm{Kh}(E_2(D_0^{+t}))
-(t^{-1}+1)q^{-1}\mathrm{Kh}(E_2(D_0^{+b})) .
\end{equation}
 \end{theorem}

\begin{remark} 
\label{remarkCvan}
\rm
We observe that for $t=-1$, we have $C(D_0^+,D^-,D^+)=0$, and so
 we recover the usual skein relation for the unnormalized Jones polynomial (and for the Jones polynomial as well):
\[
q^{-2}\hat{J}(D^+) - q^2 \hat{J}(D^-)  = (q^{-1} - q) \hat{J}(D_0^+). 
\]
\end{remark}

\begin{remark}\rm
The skein relation given in Theorem \ref{skeinth} is not local, in contrast with the usual skein relations in literature. This means that the quantity $C(D_0^+,D^-,D^+)$  depends on all remaining diagrams $D_0^+,D^-,D^+$.
\end{remark}

\subsection{Examples}\label{sec-Examples}
In this subsection we will apply Theorem \ref{skeinth} to the Hopf link and the left-handed trefoil knot.

\subsubsection{The Hopf Link}
\label{HopfExample}
\begin{figure}[h]
\includegraphics[scale=0.7]{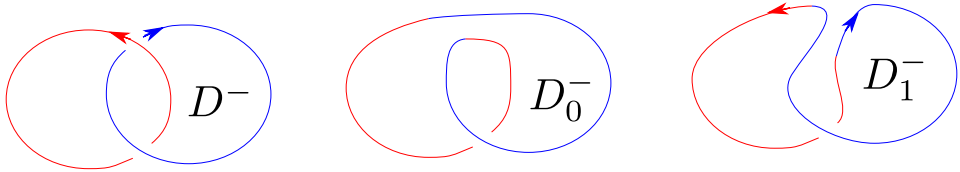}
\caption{Hopf link with two negative crossings and resolution of one crossing \label{fig:HopfLink}}
\end{figure}
First we will study the Hopf link with two negative crossings (Figure \ref{fig:HopfLink}). 

Using the invariance of classical Khovanov theory under Reidemeister moves we see that $D_0^+=D_1^-$ has the Khovanov homology of the unknot:
\[
KH^{0,1}(D_0^+)= \langle e\otimes e \rangle \cong \Q \quad \text{ and } \quad
KH^{0,-1}(D_0^+)= \langle x \otimes e \rangle \cong \Q.
\]

The Khovanov homology of $D^-$ is given in \cite[Example 3.7]{Turner2006-ck}, while $D^+$ is the union of two unlinked unknots. The Khovanov homology of the two links is  given in the following tables:
\begin{center}
$KH(D^-)$:
\newcolumntype{g}{>{\columncolor{Gray}}c}
\begin{tabular}{|g|c|c|c|}
\hline
\rowcolor{LightCyan}
$j \backslash i$ & -2  & 0 \\
\hline
0   &    &     $\mathbb{Q}$ \\
-2  & &    $\mathbb{Q}$  \\
-4  & $\mathbb{Q}$     & \\
-6  &  $\mathbb{Q}$   &   \\
\hline
\end{tabular}
\qquad $KH(D^+)$:
\begin{tabular}{|g|c|}
\hline
\rowcolor{LightCyan}
$j \backslash i$ &  0 \\
\hline
2  & $\mathbb{Q}$ \\
0  &  $\mathbb{Q}^2$ \\
-2 &  $\mathbb{Q}$\\
\hline
\end{tabular}
\end{center}

We thus have:\\
\begin{center}
\newcolumntype{g}{>{\columncolor{Gray}}c}
\begin{tabular}{|g|c|c|c|c|}
\hline
\rowcolor{LightCyan}
$(i,j)$ & $KH^{i-2,j-3}(D_0^+)$ &
$KH^{i-2,j-4}(D^-)$ &
$KH^{i,j}(D^+)$ &
$KH^{i,j-1}(D_0^+)$\\
\hline
(0,2) & &   &  $\mathbb{Q}$ & $\mathbb{Q}$ \\
(0,0) & & $\mathbb{Q}$ & $\mathbb{Q}^2$ & $\mathbb{Q}$\\
(0,-2) &  & $\mathbb{Q}$ &  $\mathbb{Q}$ & \\
(2,4)& $\mathbb{Q}$ & $\mathbb{Q}$ &   &  \\
(2,2)& $\mathbb{Q}$ & $\mathbb{Q}$ &   &  \\
\hline
\end{tabular}
\end{center}
and the following exact sequence, given in \eqref{seqCom}, is exact for every $i,j$:
\[
0 \rightarrow 
KH^{i-2,j-3}(D_0^+)
\rightarrow
KH^{i-2,j-4}(D^-)
\rightarrow
KH^{i,j}(D^+)
\rightarrow
KH^{i,j-1}(D_0^+)
\rightarrow 
0.
\]
Therefore, in this example, the skein relation \eqref{skein-rel-bar} for the Khovanov polynomial holds with  $C(D_0^+,D^-,D^+)=0$.

\subsubsection{The left-handed trefoil knot}
In this example we will study the trefoil knot $D^-$ with three negative crossings.
Then $D^+$ is the unknot and $D_0^+$ is the Hopf link. The Khovanov homology of the three links is given in the following tables: 
\begin{center}
$KH(D^-)$:
\newcolumntype{g}{>{\columncolor{Gray}}c}
\begin{tabular}{|g|c|c|c|}
\hline
\rowcolor{LightCyan}
$j \backslash i$ & 0 & -2 & -3 \\
\hline
-9   &    &  &   $\mathbb{Q}$ \\
-5  & & $\mathbb{Q}$ &    \\
-3  & $\mathbb{Q}$   &  & \\
-1  &  $\mathbb{Q}$ &  &   \\
\hline
\end{tabular}
\qquad $KH(D^+)$:
\begin{tabular}{|g|c|}
\hline
\rowcolor{LightCyan}
$j \backslash i$ &  0 \\
\hline
1  & $\mathbb{Q}$ \\
-1 &  $\mathbb{Q}$\\
\hline
\end{tabular}
\qquad
$KH(D_0^+)$:
\begin{tabular}{|g|c|c|c|}
\hline
\rowcolor{LightCyan}
$j \backslash i$ & -2  & 0 \\
\hline
0   &      &   $\mathbb{Q}$ \\
-2  &  &  $\mathbb{Q}$  \\
-4  & $\mathbb{Q}$     & \\
-6  &  $\mathbb{Q}$   &   \\
\hline
\end{tabular}
\end{center}
From the above tables we compute
\begin{align*}
\mathrm{Kh}(D^-) &=\frac{1}{t^3q^9 }+\frac{1}{t^2q^5 }+\frac{1}{q^3}+\frac{1}{q} \\
\mathrm{Kh}(D^+) &=\frac{1}{q} +q
\\
\mathrm{Kh}(D_0^+) &= \frac{1}{ t^2q^6}+\frac{1}{t^2q^4 }+\frac{1}{q^2}+1
\end{align*}

We thus have:\\
\begin{center}
\newcolumntype{g}{>{\columncolor{Gray}}c}
\begin{tabular}{|g|c|c|c|c|}
\hline
\rowcolor{LightCyan}
$(i,j)$ & $KH^{i-2,j-3}(D_0^+)$ &
$KH^{i-2,j-4}(D^-)$ &
$KH^{i,j}(D^+)$ &
$KH^{i,j-1}(D_0^+)$\\
\hline
 (-2,-3)& &   &   & $\mathbb{Q}$ \\
 (-2,-5)& &  & & $\mathbb{Q}$\\
 (-1,-5)&  & $\mathbb{Q}$ &   & \\
 (0,1)& &  &  $\mathbb{Q}$ &$\mathbb{Q}$   \\
 (0,-1)& $\mathbb{Q}$ & $\mathbb{Q}$ &$\mathbb{Q}$   & $\mathbb{Q}$ \\
  (0,-3)& $\mathbb{Q}$ &  &   &  \\
   (2,3)& $\mathbb{Q}$ & $\mathbb{Q}$ &   &  \\
  (2,1)& $\mathbb{Q}$ & $\mathbb{Q}$ &   &  \\
  \hline
\end{tabular}
\end{center}

By looking at the table above, we  only need to 
study the cases where $j \in \{-5,-3,-1,1,3\}$.\\

{\bf Case $j=-5$:}
In the first page, we have the following sequences of Khovanov homology:
\[
\xymatrix{
0 \ar[d] &  & & 0 \ar[d] \\
  0=KH^{-4,-8}(D_0^+)  \ar[d] &  & &
  0=KH^{-3,-8}(D_0^+) \ar[d] \\
  0= KH^{-4,-9}(D^-) \ar[d] & & E_{2,-5}^{-3,-9}(D^-) \ar[ldd]|{\cong} 
  \ar@{}[r]|-*[@]{\cong}&
  \mathbb{Q} \cong KH^{-3,-9}(D^-) \ar[d]  \\
  0=KH^{-2,-5}(D^+) \ar[d] &  & &
  0=KH^{-1,-5}(D^+) \ar[d] \\
  \mathbb{Q} \cong KH^{-2,-6}(D_0^+) \ar[d]
  \ar@{}[r]|-*[@]{\cong}
  & E_{2,-5}^{-2,-6}(D_0^{+b}) & & 0= KH^{-1,-6}(D_0^+) \ar[d] \\
  0 & & & 0
}
\]
The above diagonal morphism fits within the third sequence in \eqref{ses-pages} for $i=-1,j=-5$  
\[
\xymatrix{
  0 \ar[r] &
  E_{2,-5}^{-3,-9}(D^-)\cong \Q \ar[r] &
  E_{2,-5}^{-2,-6}(D_0^{+b}) \cong \Q \ar[r] &
  E_{3,-5}^{-2,-6}(D_0^{+b}) \ar[r] & 
  0
}
\]
whence $E_{3,-5}^{-2,-6}(D_0^{+b})=0$.

{\bf Case $j=-3$:}
In the first page, we have the following sequences of Khovanov homology:
\[
\xymatrix{
0 \ar[d] &  & & 0 \ar[d] 
\\
  0=KH^{-4,-6}(D_0^+)  \ar[d] &  & &
  0=KH^{-3,-6}(D_0^+) \ar[d] 
\\
  0= KH^{-4,-7}(D^-) \ar[d] & & E_{2,-3}^{-3,-7}(D^-) \ar[ldd]|{} 
  \ar@{}[r]|-*[@]{\cong}&
  0=KH^{-3,-7}(D^-) \ar[d]  
  \\
  0=KH^{-2,-3}(D^+) \ar[d] &  & &
  0=KH^{-1,-3}(D^+) \ar[d] 
  \\
  \mathbb{Q} \cong KH^{-2,-4}(D_0^+) \ar[d]
  \ar@{}[r]|-*[@]{\cong}
  & E_{2,-3}^{-2,-4}(D_0^{+b}) & & 0= KH^{-1,-4}(D_0^+) \ar[d] \\
  0 & & & 0
}
\]
The above diagonal morphism fits within the third sequence in \eqref{ses-pages} for $i=-1,j=-3$  
\[
\xymatrix{
  0 \ar[r] &
  E_{2,-3}^{-3,-7}(D^-)\cong 0 \ar[r] &
  E_{2,-3}^{-2,-4}(D_0^{+b}) \cong \Q \ar[r] &
  E_{3,-3}^{-2,-4}(D_0^{+b})  \ar[r] & 
  0
}
\]
whence $E_{3,-3}^{-2,-4}(D_0^{+b})\cong \Q$.

We also have the following sequences of Khovanov homology:
\[
\xymatrix{
0 \ar[d] &  & & 0 \ar[d] \\
  0=KH^{-3,-6}(D_0^+)  \ar[d] &  & E_{2,-3}^{-2,-6}(D_0^{+t})   \ar@{}[r]|-*[@]{\cong}\ar[ldd]|{} &
  \Q=KH^{-2,-6}(D_0^+) \ar[d] \\
  0= KH^{-3,-7}(D^-) \ar[d] & & 
&
  0=KH^{-2,-7}(D^-) \ar[d]  \\
  0=KH^{-1,-3}(D^+) \ar[d] \ar@{}[r]|-*[@]{\cong} & E_{2,-3}^{-1,-3}(D^+) & &
  0=KH^{0,-3}(D^+) \ar[d]  \\
  0 = KH^{-1,-4}(D_0^+) \ar[d]
  &  & & 0= KH^{0,-4}(D_0^+) \ar[d] \\
  0 & & & 0
}
\]
The above diagonal morphism fits within the second sequence in \eqref{ses-pages} for $i=0,j=-3$  
\[
\xymatrix{
  0 \ar[r] &
  E_{3,-3}^{-2,-6}(D_0^{+t}) \ar[r] &
  E_{2,-3}^{-2,-6}(D_0^{+t}) \cong \Q \ar[r] &
  E_{2,-3}^{-1,-3}(D^+) \cong 0 \ar[r] & 
  0
}
\]
whence $E_{3,-3}^{-2,-6}(D_0^{+t})\cong \Q$.

$ $\\
{\bf Cases $j \in \{-1,1,3\}$:} In these cases all vertical sequences are exact so the corresponding elements in the second page are all zero. 
\[
\xymatrix{
  0 \ar[d] \\
  KH^{-2,-4}(D_0^+)\cong \mathbb{Q} \ar[d]
\\
KH^{-2,-5}(D^-) \cong \mathbb{Q} \ar[d]
\\
KH^{0,-1}(D^+) \cong \mathbb{Q} \ar[d]
\\
KH^{0,-2}(D_0^+) \cong \mathbb{Q} \ar[d] \\
0
}
\qquad
\vrule
\qquad
\xymatrix{
 0 \ar[d] \\
  KH^{0,-2}(D_0^+)\cong \mathbb{Q} \ar[d]
\\
KH^{0,-3}(D^-) \cong \mathbb{Q} \ar[d]
\\
KH^{2,1}(D^+) = 0 \ar[d]
\\
KH^{2,0}(D_0^+) = 0 \ar[d] \\
0 
}
\qquad
\vrule
\qquad
\xymatrix{
 0 \ar[d] \\
  KH^{0,0}(D_0^+)\cong \mathbb{Q} \ar[d]
\\
KH^{0,-1}(D^-) \cong \mathbb{Q} \ar[d]
\\
KH^{2,3}(D^+) = 0 \ar[d]
\\
KH^{2,2}(D_0^+) = 0 \ar[d] \\
0 
}
\]
There is also a sequence where the Khovanov homology of $D^-$ is equal to $0$:
\[
  \xymatrix{
 0 \ar[d] \\
 KH^{-2,-2}(D_0^+)= 0  \ar[d]
\\
 KH^{-2,-3}(D^-) = 0 \ar[d]
\\
 KH^{0,1}(D^+) \cong \mathbb{Q} \ar[d]
\\
 KH^{0,0}(D_0^+) \cong \mathbb{Q} \ar[d] \\
 0 
}
\]

We now compute:
$$
\mathrm{Kh}(E_{2}(D_0^{+b}))=\frac{1}{t^2q^6} + \frac{1}{t^2q^4}
\quad \text{and}\quad
\mathrm{Kh}(E_{2}(D_0^{+t}) )= \frac{1}{t^2q^6}
$$
and so 
\begin{eqnarray*}
C(D_0^+,D^-,D^+)&=&(t+1)q \left(tq^2\mathrm{Kh}(E_{2}(D_0^{+t}) )-
\mathrm{Kh}(E_{2}(D_0^{+b})\right)\\
&=& (t+1)q\left( \frac{1}{tq^4} -\frac{1}{t^2q^6} - \frac{1}{t^2q^4}\right)\\
&=& (t+1)\left( \frac{tq^2-1-q^2}{t^2q^5}\right).
\end{eqnarray*}
Moreover, we have 
\begin{eqnarray*}
(t^2q^3 -q) \mathrm{Kh}(D_0^+)
 -
 t^2q^4
\mathrm{Kh}(D^-) +
\mathrm{Kh}(D^+)&=&
\frac{1}{q^3}+\frac{1}{q}+t^2q+t^2q^3-\frac{1}{t^2q^5}-\frac{1}{t^2q^3}-\frac{1}{q}-q\\
& &-\frac{1}{tq^5}-\frac{1}{q}-t^2q-t^2q^3+\frac{1}{q}+q\\
&=& \frac{1}{q^3}-\frac{1}{t^2q^5}-\frac{1}{t^2q^3}-\frac{1}{tq^5}\\
&=& \frac{t^2q^2-1-q^2-t}{t^2q^5}\\
&=& (t+1)\left( \frac{tq^2-q^2-1}{t^2q^5}\right).
\end{eqnarray*}
Therefore, Equation \eqref{skein-rel-bar} is satisfied.

\section{A categorification of the $\theta$-invariant}

\subsection{The $\theta$-invariant} Our initial motivation for looking for a skein relation for the Khovanov homology was our desire to categorify the $\theta$-invariant, a skein link invariant that generalizes the Jones polynomial.  The $\theta$-invariant is a $2$-variable polynomial invariant defined in \cite{GoLa}, and is obtained as a specialization of the $3$-variable polynomial skein link invariant $\Theta$ introduced in \cite{CJKL}, in the same way that the Jones polynomial is obtained as a specialization of the HOMFLYPT polynomial. The invariants $\Theta$ and $\theta$ are generalizations of invariants obtained from Markov traces on  the Yokonuma--Hecke algebra of type $A$ and the Framization of the Temperley--Lieb algebra respectively using Jones's method (see \cite{CJKL, GJKL} for these algebraic constructions).

Let $\mathcal{L}$ denote the set of oriented links. Let $L \in \mathcal{L}$ with components $K_1, \ldots K_r$ ($r \geq 1$). Given a diagram $D$ of $L$, we will call a crossing \emph{mixed} if it is between two different components of $D$. We will write
$L=\sqcup_{i=1}^r K_i$ if there exists a diagram of $L$ without mixed crossings, that is, $L$ is a union of $r$ unlinked knots. 
We can define $\theta$ as follows \cite[Theorem 1]{GoLa}:

\begin{theorem}
Let $q,E$ be indeterminates. There exists a unique ambient isotopy invariant
$$\theta: \mathcal{L} \rightarrow \C[q^{\pm 1}, E^{\pm 1}]$$ defined by the following rules:\\
(a) For all $r \geq 1$, we have 
\begin{equation}\label{theta1}
\theta(\sqcup_{i=1}^r K_i) = E^{1-r} J(\sqcup_{i=1}^r K_i).
\end{equation}
(b) On mixed crossings the skein relation of the Jones polynomial holds, that is,
\begin{equation}\label{theta2}
 q^{-2}\theta(L_+) - q^2 \theta(L_-)=(q^{-1}-q)\theta(L_0)
 \end{equation}
where $L_+,L_-,L_0$ is a Conway triple.
\end{theorem}

\begin{remark}\rm
The skein relation \eqref{theta2} is actually not the same as the one of  \cite[Theorem 1]{GoLa}. We have changed it slightly, replacing $q$ by $-q$, in order to be in agreement with the skein relation satisfied by the Jones polynomial in this paper.
\end{remark}

\begin{remark}\rm
We have $J(\sqcup_{i=1}^r K_i) = (q^{-1}+q)^{r-1} \prod_{i=1}^r J(K_i)$.
\end{remark}

\begin{remark}\rm
Note that in $L_0$ the two components involved in the mixed crossing have been fused into one.
\end{remark}

The invariant $\theta$ is a specialization of the $3$-variable polynomial invariant, first introduced in \cite[Theorem 8.1]{CJKL}, 
$\Theta: \mathcal{L} \rightarrow \C[q^{\pm 1}, E^{\pm 1}, {\mu}^{\pm 1}]$, which is  is defined in the same way as $\theta$, but satisfying the skein relation
$$ {\mu}^{-1}\Theta(L_+) - {\mu} \Theta(L_-)=(q^{-1}-q)\theta(L_0). $$
Taking $\mu=q^2$ yields $\theta$.

The existence of the invariant $\Theta$ is proved in \cite{CJKL} by showing that it coincides with a variation of a  $3$-variable invariant for tied links defined by Aicardi and Juyumaya in \cite{AiJu1, AiJu2}. Another diagrammatic proof of tis existence is given by Kauffman and Lambropoulou in \cite{Kauffman2017-ra}, using the notion of ``descending stacks''. We will make use of them when we define the categorification of $\theta$ in the next sections.

Let $d \in \mathbb{N}^*$. For $E=1/d$, the invariants $\Theta$ and $\theta$ can be defined algebraically as Markov traces on the Yokonuma--Hecke algebra of type $A$ and the Framization of the Temperley--Lieb algebra respectively, using Jones's technique. This algebraic construction is the third proof of their existence. 
For the rest of the paper, we will only be interested in this case, that is, when $E$ is the inverse of a positive integer $d$. For $d=1$, the invariants $\Theta$ and $\theta$ coincide  with the HOMFLYPT and the Jones polynomial respectively. In general, 
by \cite[Theorem 8.2]{CJKL} and \cite[Theorem 5]{GoLa} (see also \cite[Example 4.16]{Chlou}),  we have:

\begin{theorem}
The invariants $\Theta$ and $\theta$ are stronger than the HOMFLYPT and the Jones polynomial respectively.
\end{theorem}

By ``stronger'', we mean that these invariants distinguish links that the others (the ``weaker'' ones) cannot distinguish. 

\subsection{Knot theory as dynamical system}

\label{sec:knottheory-invariant}

In this section, we will describe the algorithm of \cite{Lickorish1987-am}  and \cite{Kauffman2017-ra} for associating to each oriented link diagram a family of unions of unlinked knots, called  ``descending stacks''. This will allow us to compute the value of any map $f$ on $\mathcal{L}$ that satisfies a generalized skein relation via a recursive process resembling a dynamical system whose initial condition is the value of $f$ on any union of unlinked knots. However, constructing the descending stacks depends on several choices that we make, so in order to prove the $f$ is well-defined, we have to further show that the value of $f$ on an oriented link does not depend on the choices made.

Let $L$ be an oriented link diagram. We say that $L$ is 
\emph{ordered}, if an order is given to its components, and 
\emph{based}, if a 
basepoint is chosen on each component.
 If $L$ is both ordered and based, we say that $L$ is \emph{generic}.
We can turn every oriented link diagram to a generic one by making the needed choices. Of course, the associated generic diagram is not unique.

A generic diagram is a {\it  descending stack} if, when walking along its components in their given order following their orientations and starting from their basepoints, every mixed crossing is first traversed along its over-arc.  The structure of a descending stack depends on the ordering of its components, but not on the choice of basepoints or the numbering of the mixed crossings. 

Starting from a generic link diagram $L$, with components $K_1,K_2,\ldots,K_r$, we can associate to it a family of descending stacks as follows:\smallbreak

\begin{enumerate}
[{\bf Step 1.}]

\item {} We perform the following procedure for $i=1$, then $i=2$, $i=3$, until we reach $i=r$: We start  walking from the basepoint on $K_i$ following its orientation and every time we come across a  mixed crossing  along its under-arc, we mark it. We continue until we return to the basepoint. Note that if all marked mixed crossings so far are switched, we obtain a generic diagram with $r$ components, where $K_1,\ldots,K_i$ are unlinked from the remaining components and lie above them, and $K_1$ lies above $K_2$, $K_2$ lies above $K_3$, etc. 

 \smallbreak

\item {}  We proceed with replacing our initial diagram by two new generic diagrams as follows: The first one, denoted by $L_1$, is obtained by switching the first  marked mixed crossing of the previous step, and thus has $r$ components. The second one, denoted by $L^{(1)}$, is obtained by smoothing the first marked mixed crossing of Step 1, and thus has $r-1$ components (two of the original components are fused into one). These two new diagrams are made generic by the same choices as for $L$. For a component resulting from the merging of two, we choose as basepoint the one of the smaller in order component involved. 
We repeat the same procedure on the second marked mixed crossing of $L_1$,
and obtain two diagrams $L_2$ and $L^{(2)}$. We continue until have done the same thing for all marked mixed crossings of Step 1. If $s$ is the total number of such mixed crossings, then we end up with an $r$-component link $L_s=:\delta L$, which is a descending stack with components $K_1, K_2,\ldots,K_r$, and  $s$ $(r-1)$-component links $L^{(1)},L^{(2)},\ldots,L^{(s)}$. We define $s$ to be the
\emph{distance} of $L$ from  $\delta L$. Clearly, the distance of a generic diagram is well-defined.

 \smallbreak

\item {} We apply the above procedure (Steps 1 and 2) for each generic diagram $L^{(1)},L^{(2)},\ldots,L^{(s)}$ obtained through the smoothing of mixed crossings of $L$.

\end{enumerate} 

We repeat Steps 1--3 until we end up with a family of descending stacks
$\{\delta_1:=\delta L, \delta_2,\ldots,\delta_m\}$. Among them only $\delta_1$ has $r$ components, while all other descending stacks have less than $r$ components.

Let now $R$ be an integral domain.  We say that
$f: \mathcal{L} \rightarrow R$ satisfies a generalized skein relation if
\begin{equation} \label{generalizedSkein}
r_+ f(L_+) + r_- f(L_-)+ r_0 f(L_0)+r_\infty(L_+, L_-,L_0)=0, 
\end{equation}
where $r_+, r_- \in R^\times$, $r_0,r_\infty \in R$ and
$L_+,L_-,L_0$ is a Conway triple.
In a classical skein relation, we have $r_\infty(L_+, L_-,L_0)=0$, but the above expression allows us to handle also the case of the Khovanov homology.

Let  $L$ be a generic link diagram, and let $\Delta=\{\delta_1, \delta_2,\ldots,\delta_m\}$ be the family of descending stacks associated to it. 
Applying the generalized skein relation to every marked mixed crossing 
that we switched or smoothed to reach $\Delta$ (let us denote their number by $M$) yields
\begin{equation}\label{alphainf}
f(L) = \sum_{i=1}^m \alpha_i(r_+,r_-,r_0) f(\delta_i) +\alpha_\infty(r_\infty),
\end{equation}
where the coefficients $\alpha_i(r_+,r_-,r_0)$ are products of  $\pm r_+^{\pm 1}$, 
$\pm r_-^{\pm 1}$, $\pm r_0$, for $i=1,\ldots,m$, while $\alpha_\infty(r_\infty),$ is a product of $M$ terms of the form 
$-r_\infty(c)$, where $c$ runs over Conway triples, multiplied with either $r_+^{-1}$ or $r_-^{-1}$.
Therefore, if we have an ``inital condition'' that gives the value of $f$ on any union of unlinked knots, then we can compute $f(L)$. 

Now, the map $f$ is well-defined if whenever $L$ and $L'$ are two generic diagrams of the same oriented link in $\mathcal{L}$, we have $f(L)=f(L')$. 
In order to show that $f$ is well-defined, it is enough to prove that its value on an oriented link $L$ does not depend on the sequence of mixed crossing switches, nor the choice of basepoints, nor the ordering of its components.
Link invariants that satisfy skein relations or generalized skein relations, such as the Jones polynomial, the $\theta$-invariant or the Khovanov homology, are obviously well-defined.

\subsection{Categorifying $\theta$: the case of knots}

Let $d \in \N^*$ and $E=1/d$.
As in the case of the Jones polynomial, we would like to construct a homology whose Poincar\'e polynomial, evaluated at $t=-1$, yields the \emph{unnormalized $\theta$-invariant} $\hat{\theta}$, which is given by $\hat{\theta}(L)=d(q+q^{-1})\theta(L)$ on any oriented link $L$. Further, for $d=1$, it should become the classical Khovanov homology.

Let now $G$ be the cyclic group of order $d$. For any knot $K$, we define
\begin{equation}\label{first step}
  KH^{i,j}_d(K):=\Q[G]\otimes_{\Q} KH^{i,j}(K).
\end{equation}
The corresponding Poicar\'e polynomial is
\[
  \Kh_d(K)=\sum_{i} t^i \mathrm{qdim}(KH^{i,*}_d(K))=d \cdot \Kh(K).
\]

If now $L_1$, $L_2$ are unions of unlinked knots, then we can define inductively 
$KH^{i,j}_d$ on $L_1 \sqcup L_2$ as follows:
\begin{equation}\label{second step}
  KH^{i,j}_d(L_1 \sqcup L_2):=\bigoplus_{\nu_1 +\nu_2=i \atop
  \mu_1+\mu_2=j}  KH^{\nu_1,\mu_1}_d(L_1) \otimes KH^{\nu_2,\mu_2}_d(L_2).
\end{equation}
The above result implies the multiplicativity of the corresponding Poincar\'e polynomial:
\[
  \mathrm{Kh}_d(L_1\sqcup L_2)=\mathrm{Kh}_d(L_1) \cdot \mathrm{Kh}_d(L_2).
\]

\begin{remark}\rm
For $d=1$, and thus the classical Khovanov homology, one can prove \eqref{second step} using the  K\"unneth spectral sequence \cite[Theorem 3.6.1]{Weibel},
after observing that the Khovanov complex of $L_1 \sqcup L_2$ is the tensor product of the Khovanov complexes of $L_1$ and $L_2$. It is important that we are working over the field $\mathbb{Q}$ for this to hold. The case of Khovanov homology with coefficients in $\Z$ or, even worse, over an arbitrary ring is much more complicated. 
\end{remark}

We deduce that if $K_1,\ldots,K_r$ are knots, then
\begin{equation} \label{desc-stack-def}
\Kh_d(\sqcup_{i=1}^r K_i)=\prod_{i=1}^r \Kh_d(K_i) = \prod_{i=1}^r d\Kh(K_i)=
d^r \Kh(\sqcup_{i=1}^r K_i).
\end{equation}
Evaluating  the above formula at $t=-1$ yields:
\[
\Kh_d(\sqcup_{i=1}^r K_i)\mid_{t=-1} = d^r \Kh(\sqcup_{i=1}^r K_i)\mid_{t=-1}=
d^r \hat{J}(\sqcup_{i=1}^r K_i) = 
\hat{\theta}(\sqcup_{i=1}^r K_i) .
\]

\subsection{Categorifying $\theta$: the case of links}\label{3.3}

Equation \eqref{desc-stack-def} is the categorified equivalent of Equation \eqref{theta1}. We would now like to obtain the value of $\Kh_d$ on any link $L$ by applying the following rule:  On mixed crossings the skein relation of the Khovanov polynomial holds. However, the Khovanov polynomial does not satisfy a skein relation in the classical sense; the closest to a skein relation that we have is Equation \eqref{skein-rel-bar2}, which can be rewritten as follows (replacing $\Kh$ with $\Kh_d$):
\begin{equation}
\label{skeinSwap}
\Kh_d(L) =   y^{2 \epsilon_P} \Kh_d(\sigma_P L) - 
\epsilon_P y^{\epsilon_P} z \Kh_d( s_P L) +  C'(L,\sigma_P L, s_PL)
\end{equation} 
where $z=tq-(tq)^{-1}$,  $y=tq^2$, $\sigma_P L$ is the link with a given crossing $P$ switched
(that is, $L_+$ for $L_-$ and vice versa), $s_P L$ is the link with the crossing  $P$ smoothed, $\epsilon_P=\pm 1$ is the sign of the crossing $P$
and $C'(L,\sigma_P L, s_PL)= \epsilon_P y^{\epsilon_P} C^{\rm sym}(s_PL, L_-,L_+)\in \Q[q^{\pm 1},t^{\pm 1}]$.
We say that $C'(L,\sigma_P L, s_PL)$ is the \emph{defect} of the skein relation.
Unfortunately, $C'(L,\sigma_P L, s_PL)$
depends on all three diagrams $L$, $\sigma_P L$ and $s_PL$. Because of our lack of control on the value of the defect, we were not able to show that, by applying rule \eqref{skeinSwap} on mixed crossings, we obtained a well-defined link invariant. This is the reason why we decided to introduce and extra framing variable $d'$ and 
 the following variation of the generalized skein relation
\begin{equation}
\label{skeinSwap2}
\boxed{
\Kh_{d,d'}(L) =   
y^{2 \epsilon_P} \Kh_{d,d'}(\sigma_P L) - 
\epsilon_P y^{\epsilon_P} z 
d'
\Kh_{d,d'}( s_P L) +  d^{\alpha(L)} C'(L,\sigma_P L, s_PL)
}
\end{equation}
where $\alpha(L)$ is the number of components of the link $L$. We can now construct a link invariant as follows:

We start with an oriented link $L$  with components $K_1,\ldots,K_r$. 
There are $r!$ ways of ordering the components of $L$, each corresponding to 
a  permutation $\beta \in \mathfrak{S}_r$. We will consider the ordering as part of the structure of $L$ and we will denote the link $L$ with the ordering $\beta$ by $L^\beta$.

Now, let $\beta \in \mathfrak{S}_r$ 
and let  $\Gamma_{L^\beta}$  
denote the set of  diagrams representing $L^\beta$ that have minimal number of crossings. It is clear that any diagram in this set can be transformed to another one in this set by  a sequence of Reidemeister moves.  
For every diagram $D \in \Gamma_{L^\beta}$, we denote by 
$\mathrm{mix}(D)$ the set of its mixed crossings and by 
$\mathrm{Cmix}(D)$ the subset of $\mathrm{mix}(D)$ of crossings at which the component below is smaller than the component above with respect to the $\beta$-ordering. Therefore, the set $\mathrm{Cmix}(D)$ consists of the mixed crossings that  need to be switched so that in the descending stack $\delta L^\beta$ the order of the resulting  knot components from top to bottom is given by the $\beta$-ordering (with the smallest component on top).

Let $D \in \Gamma_{L^\beta}$.
For any mixed crossing $P \in \mathrm{Cmix}(D)$, we can perform a switching and smoothing according to the relation given in Equation \eqref{skeinSwap2}. 
More specifically, we have (we change the notation to $\Kh_{d,d'}^P$ in order to keep track of the crossing $P$):
\[
\Kh_{d,d'}^P(D) =   
y^{2 \epsilon_P} \Kh_{d,d'}^P(\sigma_P D) - 
\epsilon_P y^{\epsilon_P} z 
d'
\Kh_{d,d'}^P(s_P D) +  d^{\alpha(L)} C'(D,\sigma_P D, s_PD).
\]
We observe that the number of components $\alpha(L)$ of $L$ does not depend on the choice of the diagram $D$.
As far as the ordering is concerned, the link $\sigma_P D$, where the components at the mixed crossing $P$ are switched, inherits the ordering of $L^\beta$, while in $s_P D$, where the components at the mixed crossing $P$ are merged together, 
the new component inherits its numbering from the smallest of the two components involved; the ordering of the components of $s_P D$ corresponds to a permutation
$\beta' \in \mathfrak{S}_{r-1}$ which respects otherwise the $\beta$-ordering of the components of $L$. 
It is clear that repeating this procedure will result in descending stacks, where $\Kh_{d,d'}$ can be defined with the use of Equation \eqref{desc-stack-def} (replacing $\Kh_d$ with $\Kh_{d,d'}$). Indeed, the diagram $\sigma_P D$ has smaller distance from the descending stack $\delta L^\beta$ than the original diagram $D$ and the diagram $s_P D$ has a mixed crossing less than  $D$.

If now $L$ is a union of $r$ unlinked knots, then we define $\Kh_{d,d'}(L):=\Kh_d(L)$. Otherwise, we have
$\#\mathrm{Cmix}(D) \neq 0$ for all $D \in \Gamma_{L^\beta}$ with $\beta \in \mathfrak{S}_r$, and
we define   
\begin{equation}
\label{Khddprime}
\Kh_{d,d'}(L):=
\frac{1}{r!} 
\sum_{\beta \in \mathfrak{S}_r}
\frac{1}{\# \Gamma_{L^\beta}} 
\sum_{D \in \Gamma_{L^\beta}}
\frac{1}{\# \mathrm{Cmix}(D)} 
\sum_{P \in \mathrm{Cmix}(D)}
\Kh_{d,d'}^P(D)
\end{equation}
This is essentially a quantity which is summed  over all possible choices we made on the level of the generic diagram, hence it is independent of them. It is also clear that, thanks to the use of the set $\Gamma_{L^\beta}$, the value of  $\Kh_{d,d'}(L)$ is invariant under Reidemeister moves. Hence, $\Kh_{d,d'}$ is a well-defined link invariant.

We will now see why we consider the link invariant $\Kh_{d,d'}$ a categorification of $\theta$. First, we observe that for $d'=1$ and $t=-1$, Equation
\eqref{skeinSwap2} 
becomes
\[
\Kh_{d,1}(L) =   
y^{2 \epsilon_P} \Kh_{d,1}(\sigma_P L) - 
\epsilon_P y^{\epsilon_P} z 
\Kh_{d,1}( s_P L) 
\]
with $z=q^{-1}-q$ and $y=-q^2$, which is the same skein relation satisfied by $\theta$ on mixed crossings. Since $\theta$ is a well-defined link invariant, it is 
invariant under all possible choices. So in Equation \eqref{Khddprime}, we sum over all possible choices the same value and obtain
\[
\xymatrix{
\Kh_{d,d'}(L) \ar@{|->}[rr]^-{t=-1,d'=1} &  &\theta(L).
}
\]

Let us now consider the case $d'=d$. 
We extend the definition of $\Kh_d$ on any link $L$ by establishing the following rule: On mixed crossings $\Kh_d$ satisfies the generalized skein relation
\[
\Kh_{d}(L) =   
y^{2 \epsilon_P} \Kh_{d}(\sigma_P L) - 
\epsilon_P y^{\epsilon_P} z 
d
 \Kh_{d}( s_P L) +  d^{\alpha(L)} C'(L,\sigma_P L, s_PL).
\]
We will show that $\Kh_d$ is well-defined; in particular, we have 
$\Kh_d(L)=d^{\alpha(L)}\Kh(L)$. 
We deduce that
\[
\xymatrix{
\Kh_{d,d'}(L) \ar@{|->}[rr]^-{d=d'} &  & \Kh_d(L).
}
\]

We start with a generic diagram of an oriented link $L$ with $r$ components. 
We follow the algorithm described in \S \ref{sec:knottheory-invariant}
 in order to form the family of descending stacks $\{\delta_1,\ldots,\delta_m\}$ associated to it.  
We can  draw a graph, called the \emph{skein tree} of $L$, whose vertices are labelled by the link diagrams appearing when applying the algorithm (starting from $L$ at the top and resulting to the descending stacks $\delta_1,\ldots,\delta_m$ at the bottom) and whose edges connect any link diagram with its switched and smoothed versions at a given mixed crossing. Each vertex of the skein tree is the outcome of a series of switchings and smoothings, starting from the link diagram $L$.

Assume that we are 
in the vertex $L_v$ of the skein tree of $L$, and that $L_v$ is obtained from $L$ after applying
$r_v$ switchings and  $r'_v$ smoothings.
This link $L_v$ has $r-r'_v$ components and its contribution $\Kh_d(L_v)$ in the computation of $\Kh_d(L)$ is multiplied by $d^{r'_v}$. 
Therefore, by construction, the ``total defect'' in the computation of $\Kh_d(L)$ is given by 
\[
\sum_{v} d^{r-r_v'} d^{r'_v} C'(L_v, \sigma_{P_v} L_v, s_{P_v} L_v), 
\]
where $P_v$ is the mixed crossing that is switched and smoothed at the vertex $L_v$ of the skein tree according to the algorithm.  
It is clear that $\Kh_d(L)=d^r\Kh(L)$.

We note that, in order to be able to obtain both $\theta$ and  $\Kh_d$ as specializations of the link  invariant $\Kh_{d,d'}$, we could not have simply imposed the rule \eqref{skeinSwap2} on mixed crossings, because the discontinuity of the function $\Kh_{d,d'}$ would not allow us to show invariance under Reidemeister moves. This is why we defined $\Kh_{d,d'}$ by summing over all possible choices as in Equation \eqref{Khddprime}.

\subsection{Open questions}

As mentioned in the introduction and in the previous section, it is an open question whether we could have applied the method of Lickorish--Millett and Kauffman--Lambropoulou in order to define a framization of the Khovanov homology with the use of the  generalized skein relation \eqref{skeinSwap}, which is the one satisfied by the classical Khovanov homology with scalars extended to $\mathbb{Q}[G]$. It is certain that we cannot imitate their proof and show that our function $\Kh_d$ is independent of all choices made by relying on the properties of the base invariant involved. 
This is why we choose to (a) sum over all possible choices and divide by their number, and (b) introduce the framing parameter $d'$.

Then one has to prove that this new function is indeed a link invariant by showing stability under the Reidemeister moves. The Reidemeister I move poses no problem, since it is handled by the base invariant at each individual component. However, stability under the  Reidemeister II move does not seem to hold, since the proofs found in \cite{Lickorish1987-am} and in \cite{Kauffman2017-ra} use the independence of the choices made and such a treatment is not possible in our case. 
So instead of working with an arbitrary diagram representing a link $L$, we work with  diagrams with minimal number of crossings. 
In this way, we arrive at the invariant $\Kh_{d,d'}$ defined by Equation \eqref{Khddprime}.

The drawback of our method is that it seems very difficult to compute the value of the invariant $\Kh_{d,d'}$. 
Given a link $L$, there is no simple way to  find all diagrams with minimal number of crossings. Moreover, assuming that we have a diagram of $L$ with minimal number of crossings, it seems equally difficult to find all other diagrams which represent  $L$ and still have the same minimal number of crossings. 

We believe that it is interesting, in order to address the above problem, to find a way of constructing all link diagrams with a given number of minimal crossings. The existing lists, for example in the ``Knot Atlas'' \cite{knotatlas}, give information of only one representative of the Reidemeister equivalence class. If such a set $\Gamma_L$ of minimal crossings diagrams for a given link $L$ is provided, then summing 
a well-defined function $f$ on the set of link diagrams over $D\in\Gamma_L$ provides us with a link invariant $h_f$, that is,
\[
h_f(L):=\sum_{D\in \Gamma_L} f(D). 
\]
In this article we have selected the appropriate function to give us the $\theta$-invariant. An interesting example for a function to integrate in this way in a future work is the $\zeta$-function of graphs (cf.~\cite{terras}). We would like also to investigate further properties of these ``integrals'', such as the existence of a generalized skein relation.

\end{document}